%% file: main.tex
\documentclass[runningheads]{llncs}
\include{header}

\title{Decentralized Strongly-Convex Optimization with Affine Constraints: Primal and Dual Approaches\thanks{The work of D. Yarmoshik in Sections \ref{sec:introduction}, \ref{sec:locally_dual_approach}, \ref{sec:numerical_exp} was supported by the program “Leading Scientific Schools” (grant no. NSh-775.2022.1.1). The work of A. Rogozin and A. Gasnikov in Sections \ref{sec:preliminaries}--\ref{sec:globally_dual_approach} was supported by Russian Science Foundation (project No. 21-71- 30005).}}
\titlerunning{Decentralized Strongly-Convex Optimization with Affine Constraints}
\author{
    Alexander Rogozin
    \inst{1}
    \and
    Demyan Yarmoshik
    \inst{1}
    \and
    Ksenia Kopylova
    \inst{2}
    \and
    Alexander Gasnikov
    \inst{1, 3, 4}
}

\authorrunning{Rogozin et al.}
\institute{
    Moscow Institute of Physics and Technology, Moscow, Russia 
    \and
    Saint Petersburg State University, Saint Petersburg, Russia
    \and
    Caucasus Mathematical Center, Adyghe State University, Maikop, Russia
    \and
    IITP RAS, Moscow, Russia
}

\begin{document}

\maketitle
\begin{abstract}
	Decentralized optimization is a common paradigm used in distributed signal processing and sensing as well as privacy-preserving and large-scale machine learning. It is assumed that several computational entities locally hold objective functions and are connected by a network. The agents aim to commonly minimize the sum of the local objectives subject by making gradient updates and exchanging information with their immediate neighbors. Theory of decentralized optimization is pretty well-developed in the literature. In particular, it includes lower bounds and optimal algorithms. In this paper, we assume that along with an objective, each node also holds affine constraints. We discuss several primal and dual approaches to decentralized optimization problem with affine constraints.
\end{abstract}

\section{Introduction}\label{sec:introduction}

Many distributed systems such as distributed sensor networks, systems for power flow control and large-scale architectures for machine learning use decentralized optimization as a basic mathematical tool. Several applications such as power systems control \cite{molzahn2017survey,Patari2021} lead to problems where the agents locally hold optimization objectives and aim to cooperatively minimize the sum of the objectives. Moreover, every node locally holds affine constraints for its decision variable.

Decentralized optimization without affine constraints can be called a well-examined area of research. It is known that the performance of optimization algorithms executed over strongly-convex smooth objectives is lower bounded by a multiple of the graph condition number and objective condition number (up to a logarithmic factor) \cite{scaman2017optimal}. Both primal \cite{kovalev2020optimal} and dual \cite{scaman2017optimal} algorithms that reach the lower bounds have been proposed. The algorithms are based on reformulating network communication constraints as affine constraints via a communication matrix associated with the network (i.e. Laplacian matrix). Introduction of affine constraints at the nodes leads to new classes of algorithms that can be divided into two main types. The first type are consensus-based methods that can be either primal or dual \cite{dc2,dc3,pdc1,pdc2,necoara2014distributed,NECOARA2015209,NECOARA2011756}. The second type are ADMM-based methods \cite{erseghe2014distributed,rostampour2019distributed,admm1,admm2}.
Let us briefly review some of the closely related papers.

The paper \cite{necoara2014distributed} is dedicated to constrained distributed optimization and consider only separable objective functions (each agent has its own independent variable).
Moreover, affine constraints are supposed to be network-compatiable (constraint matrix can have a non-zero element on position $(i, j)$ only if there is an edge in communication graph between agents $i$ and $j$).
We do not impose such limitations: in our case each term in the objective functions depends on the same shared variable (formulation in \cite{necoara2014distributed} is obviously a special case of this) and matrix of constraints can have arbitrary structure.

In \cite{NECOARA2011756} the authors present various formulations of distributed optimization problems with different types of interconnections between constraints and objectives, including the case, when the objective (cost) cannot be represented as sum of cost functions of each agent. However, their algorithms for  problems with coupled affine constraints require to solve a ``master problem'' on central node at each iteration and thus are not decentralized.

The authors of \cite{wang2022distributed} consider multi-cluster distributed problem formulation which is a generalization of multi-agent approach.
In multi-cluster case agents within one cluster have the same decision variable while different clusters corresponds to different decision variables. 
All variables are subject to a coupled affine constraint.
By incorporating consensus constraints into dual problem with Lagrangian multipliers the author comes to solving a saddle point problem and prove asymptotic $O(1/N)$ ergodic convergence rate for their method.
Dependency of convergence rate on problem parameters in saddle point approach was  studied in \cite{yarmoshik2022decentralized}.

Our paper studies the application of different techniques to decentralized problems with affine constraints.
We obtain linear convergence rates with (explicitly specified) accelerated dependencies on function properties, constraint matrix spectrum and communication graph properties.

The paper outline is as follows. In Section \ref{sec:primal_approach} we discuss a primal approach, that is based on reformulation the initial distributed problem as a saddle-point problem and applying algorithm of paper \cite{kovalev2021accelerated} afterwards. In Section \ref{sec:globally_dual_approach}, we describe a method that allows to incorporate both affine and communication constraints to the dual function. We refer the approach in Section \ref{sec:globally_dual_approach} as a globally dual approach. Finally, in Section \ref{sec:locally_dual_approach} we describe a slightly different dual approach that firstly takes dual functions locally at the nodes and incorporates consensus constraints afterwards. We refer to the latter method as a locally dual approach.

% {\color{red}
% Distributed approaches for optimization problems with coupled constraints can be separated into two main groups: (i) primal, dual or primal-dual consensus algorithms \cite{dc2,dc3,pdc1,pdc2,necoara2014distributed,NECOARA2015209,NECOARA2011756}; (ii) ADMM-based algorithms \cite{erseghe2014distributed,rostampour2019distributed,admm1,admm2}.
% The detailed surveys on the application of distributed algorithms in power systems is given in \cite{molzahn2017survey,Patari2021}. 
% }

\section{Preliminaries}\label{sec:preliminaries}

Let $\col(x_1, \ldots, x_m)$ define a column vector of $x_1, \ldots, x_m\in\R^d$, i.e. $\col(x_1, \ldots, x_m) = [x_1^\top \ldots x_m^\top]^\top$. For matrices $P$ and $Q$, their Kronecker product is defined as $P\otimes Q$. Identity matrix of size $p\times p$ is denoted $\mI_p$. Moreover, given a symmetric positive semi-definite matrix, we denote $\lambda_{\max}(\cdot),~ \lambda_{\min}(\cdot),~ \lambda_{\min}^+(\cdot)$ its maximal, minimal and minimal nonzero eigenvalues, respectively. We also let $\sigma_{\max}(\cdot),~ \sigma_{\min}(\cdot)$ and $\sigma_{\min}^+(\cdot)$ be the maximal, minimal and minimal nonzero singular values of a matrix, respectively.

In the forthcoming analysis, we will need the following basic lemma concerning Kronecker product properties.
\begin{lemma}\label{lem:kronecker_diagonalize}
    Given two matrices $P$ and $Q$ such that $\sigma_{\min}(P) = \sigma_{\min}(Q) = 0$, we have
    \begin{align*}
        \sigma_{\max}(P\otimes\mI + \mI\otimes Q) &= \sigma_{\max}(P) + \sigma_{\max}(Q), \\
        \sigma_{\min}^+(P\otimes\mI + \mI\otimes Q) &= \min\braces{\sigma_{\min}^+(P), \sigma_{\min}^+(Q)}
    \end{align*}
\end{lemma}
\begin{proof}
Consider decompositions $P = U_P\Sigma_P V_P^\top$ and $Q = U_Q\Sigma_Q V_Q^\top$, where $U_P, V_P, U_Q, V_Q$ are orthogonal matrices and $\Sigma_P$ and $\Sigma_Q$ are diagonal matrices with corresponding eigenvalues at the diagonal. We have
\begin{align*}
    (U_P^\top\otimes U_Q^\top)(P\otimes\mI + \mI\otimes Q)(V_P\otimes V_Q) = \Sigma_P\otimes\mI + \mI\otimes\Sigma_Q.
\end{align*}
Denote singular values of $P$ as $\alpha_1, \ldots, \alpha_n$ and the singular values of $Q$ as $\beta_1, \ldots, \beta_m$. Singular values of $P\otimes\mI + \mI\otimes Q$ have form
\begin{align*}
    \lambda(\alpha_i, \beta_j) = \alpha_i + \beta_j,~ i = 1, \ldots, n,~ j = 1, \ldots, m.
\end{align*}
Therefore, $\sigma_{\max}(P\otimes\mI + \mI\otimes Q) = \sigma_{\max}(P) + \sigma_{\max}(Q)$. For the minimal nonzero singular values we obtain
\begin{align*}
    \sigma_{\min}^+(P\otimes\mI + \mI\otimes Q) = \min\braces{\sigma_{\min}^+(P), \sigma_{\min}^+(Q)}.
\end{align*}
\end{proof}

\section{Problem Statement}\label{sec:problem_statement}

Consider minimization problem with affine constraints.
\begin{align}\label{eq:problem_initial}
    &\min_{x\in\R^d}~ \sum_{i=1}^m f_i(x)~~ \text{s.t. } Bx = 0.
\end{align}
We assume that each $f_i$ is held by a separate agent, and the agents can exchange information through some communication network. Each agent also locally holds affine optimization constraints $Bx = 0$, where $B\in\R^{p\times d}$. Further we assume that $\kernel B\ne\{0\}$, because otherwise the constraints $Bx = 0$ define a set consisting of only $\{0\}$, which is not an interesting case.

We make assumptions on the optimization objectives that are standard for optimization literature \cite{nesterov2004introduction}.
\begin{assumption}
    Each $f_i$ ($i = 1, \ldots, m$) is differentiable, $\mu$-strongly convex and $L$-smooth, i.e.
    \begin{align*}
        f(y) &\geq f(x) + \angles{\nabla f(x), y - x} + \frac{\mu}{2}\norm{y - x}_2^2, \\
        f(y) &\leq f(x) + \angles{\nabla f(x), y - x} + \frac{L}{2}\norm{y - x}_2^2.
    \end{align*}
\end{assumption}
%We assume that each $f_i$ is stored at a separate computational agent, and the agents are connected by a communication network.
The communication network is represented by an undirected connected graph $\cG = (\cV, \cE)$. The communication constraints are represented by a specific matrix $W$ associated with the graph $\cG$.
\begin{assumption}\label{assum:mixing_matrix}
\item
\begin{enumerate}%[leftmargin=*]
    \item ${W}$ is a symmetric positive semi-definite matrix.
	\item (Network compatibility) For all $i,j = 1,\dots,m$ it holds $[{ W}]_{ij} = 0$ if $(i, j)\notin \cE$ and $i\ne j$.
	\item (Kernel property) For any $v = [v_1,\ldots,v_m]^\top\in\R^m$, ${ W} v = 0$ if and only if $v_1 = \ldots = v_m$, i.e.  $\kernel { W} = \spn\braces{\one}$.
\end{enumerate}
\end{assumption}
An explicit example of a matrix that satisfies Assumption \ref{assum:mixing_matrix} is the Graph Laplacian $W\in \mathbb{R}^{m\times m}$:
\begin{align}\label{eq:laplacian}
[W]_{ij} \triangleq \begin{cases}
-1,  & \text{if } (i,j) \in E,\\
\text{deg}(i), &\text{if } i= j, \\
0,  & \text{otherwise.}
\end{cases}
\end{align}

Let us introduce $\bx = \col\cbraces{x_1\ldots x_m}$ and $\mW = W\otimes\mI$. According to Assumption \ref{assum:mixing_matrix}, communication constraints $x_1 = \ldots = x_m$ can be equivalently rewritten as $\mW\bx = 0$. Also introduce $\mB = \mI\otimes B$ and $F(\bx) = \sum_{i=1}^m f_i(x_i)$. That allows to rewrite problem \eqref{eq:problem_initial} in the following way.
\begin{align}\label{eq:problem_constraints}
    \min_{\bx\in\R^{md}}~ &F(\bx) \\
    \text{s.t. } &\mW\bx = 0,~ \mB\bx = 0. \nonumber
\end{align}
Reformulation \ref{eq:problem_constraints} admits implementation of optimization methods for affinely constrained minimization. The iterations of such methods become automatically decentralized in the following sense. Let the optimization algorithm use primal or dual oracle calls of the objective function and use multiplications by the matrices representing affine constraints. In the case of problem \eqref{eq:problem_constraints} the gradient $\nabla F(\bx) = \col\sbraces{\nabla f_1(x_1)\ldots \nabla f_m(x_m)}$ is computed locally on the nodes and stored in a distributed manner across the network. Multiplication by $\mB$ is also performed locally due to its definition (i.e. the $i$-th node computes $Bx_i$), and the multiplication by $\mW$ is performed in a decentralized manner due to the network compatibility property of $W$ (see Assumption \ref{assum:mixing_matrix}).

\section{Primal Approach}\label{sec:primal_approach}

In this section, we discuss the solution of problem \eqref{eq:problem_constraints} by an algorithm APDG \cite{kovalev2021accelerated} that only uses primal oracle calls. The algorithm is designed for saddle-point problems, so we reformulate \eqref{eq:problem_constraints} as a saddle-point problem.

We add dual multipliers for the constraints and get a saddle-point problem
\begin{align}\label{eq:saddle_point_problem}
    \min_{\bx\in\R^{md}} \max_{\bu\in\R^{mp}, \bv\in\R^{md}}~ F(\bx) + \angles{\bu, \mB\bx} + \gamma\angles{\bv, \mW\bx} = 
    F(\bx) + \angles{\begin{pmatrix} \bu \\ \bv \end{pmatrix}, \begin{pmatrix} \mB \\ \gamma\mW \end{pmatrix} \bx}.
\end{align}
\begin{algorithm}[H]
	\caption{APDG: Accelerated Primal-Dual Gradient Method}
	\label{apd:alg}
	\begin{algorithmic}[1]
		\STATE {\bf Input:} $\bx^0 \in \range \mA^\top, \by^0\in \range \mA,$ $\eta_x,\eta_y,\alpha_x, \beta_x,\beta_y>0$, $\tau_x,\tau_y,\sigma_x,\sigma_y \in (0,1]$, $\theta\in(0,1)$
		\STATE $\bx_f^0 = \bx^0$
		\STATE $\by_f^0 = \by^{-1} = \by^0$
		\FOR {$k=0,1,2,\ldots$}
		\STATE $\by_m^{k} = \by^{k} + \theta(\by^{k} - \by^{k-1})$\label{apd:line:y:m}
		\STATE $\bx_g^k = \tau_x \bx^k + (1-\tau_x)\bx_f^k$\label{apd:line:x:1}
		\STATE $\by_g^k = \tau_y \by^k + (1-\tau_y)\by_f^k$
		\STATE  $\bx^{k+1} = \bx^k + \eta_x\alpha_x(\bx_g^k - \bx^{k}) 
		- \eta_x\beta_x\mA^\top\mA \bx^k  - \eta_x\left(\nabla F(\bx_g^k) 
		+ \mA^\top \by_m^{k}\right)$\label{apd:line:x:2}
		\STATE $\by^{k+1} = \by^k  - \eta_y\beta_y\mA(\mA^\top \by^k 
		+ \nabla F(\bx_g^k)) +  \eta_y  \mA \bx^{k+1}$
		\STATE $\bx_f^{k+1} = \bx_g^k + \sigma_x(\bx^{k+1} - \bx^k)$\label{apd:line:x:3}
		\STATE $\by_f^{k+1} = \by_g^k + \sigma_y(\by^{k+1} - \by^k)$
		\ENDFOR
	\end{algorithmic}
\end{algorithm}
Denote $\mA = \begin{pmatrix} \mB \\ \gamma\mW \end{pmatrix}$. In order to get complexity bounds for APDG applied to problem \eqref{eq:saddle_point_problem}, we need to bound the spectrum of $\mA$. Note that $\mA^\top\mA = \mB^\top\mB + \gamma^2\mW^2 = \mI_m\otimes (B^\top B) + \gamma^2 W^2\otimes\mI_{d}$. By Lemma \ref{lem:kronecker_diagonalize} we have
\begin{align*}
\lambda_{\max}(\mA^\top\mA) &= \lambda_{\max}(B^\top B) + \gamma^2\lambda_{\max}^2(W), \\
\lambda_{\min}^+(\mA^\top\mA) &= \min\braces{\lambda_{\min}^+(B^\top B), \gamma^2(\lambda_{\min}^+(W))^2}.
\end{align*}
% \begin{proof}
% Note that $\mA^\top\mA = \mB^\top\mB + \gamma^2\mW^2 = \mI_m\otimes (B^\top B) + \gamma^2 W^2\otimes\mI_{d}$. Consider decompositions $B^\top B = U_B\Sigma_B U_B^\top$ and $W^2 = U_W\Sigma_W U_W^\top$, where $U_B$ and $U_W$ are orthogonal matrices and $\Sigma_B$ and $\Sigma_W$ are diagonal matrices with corresponding eigenvalues at the diagonal. We have
% \begin{align*}
%     (U_W^\top\otimes U_B^\top)(\mI_m\otimes (B^\top B) + \gamma^2 W^2\otimes\mI_{d})(U_W\otimes U_B) = \mI_m\otimes\Sigma_B + \gamma^2\Sigma_W\otimes\mI_d.
% \end{align*}
% Denote the eigenvalues of $B^\top B$ as $\beta_1^2, \ldots, \beta_{d}^2$ and the eigenvalues of $W^2$ as $w_1^2, \ldots, w_m^2$. The eigenvalues of $\mB^\top\mB + \gamma^2\mW^2$ have form
% \begin{align*}
%     \lambda(\beta_i^2, w_j^2) = \beta_i^2 + \gamma^2 w_j^2,~ i = 1, \ldots, d,~ j = 1, \ldots, m.
% \end{align*}
% Therefore, $\lambda_{\max}(\mA^\top\mA) = \lambda_{\max}(B^\top B) + \gamma^2\lambda_{\max}^2(W)$. %\max_i\beta_i^2 + \max_i\gamma^2 w_i^2
% Note that exactly one $w_i^2$ equals zero and at least one of $\beta_j^2$ equals zero (otherwise, we would have $Bx = 0 \Rightarrow B^\top B x = 0 \Rightarrow x = 0$, which is not interesting)
% \begin{align*}
%     \lambda_{\min}^+(\mA^\top\mA) = \min\braces{\lambda_{\min}^+(B^\top B), \gamma^2(\lambda_{\min}^+(W))^2}
% \end{align*}
% \end{proof}
We can also compute the condition number of $\mA^\top\mA$:
\begin{align*}
    \chi(\mA^\top\mA) = \frac{\lambda_{\max}(\mA^\top\mA)}{\lambda_{\min}^+(\mA^\top\mA)} = \frac{\lambda_{\max}(B^\top B) + \gamma^2\lambda_{\max}^2(W)}{\min\braces{\lambda_{\min}^+(B^\top B), \gamma^2(\lambda_{\min}^+(W))^2}}.
\end{align*}
By accurately choosing factor $\gamma$, we can control the condition number $\chi(\mA^\top\mA)$. 
The minimal value of $\chi(\mA^\top\mA)$ is attained at $\gamma^2 = \frac{\lambda_{\min}^+(B^\top B)}{(\lambda_{\min}^+(W))^2}$ and equals $\chi(\mA^\top\mA) = \chi(B^\top B) + \chi^2(W)$.
Therefore, if we apply APDG directly to problem \eqref{eq:problem_constraints}, the complexity would be
\begin{align*}
    O\cbraces{\max\cbraces{\sqrt{\chi^2(W) + \chi(B^\top B)} \sqrt{\frac{L}{\mu}}, \chi^2(W) + \chi(B^\top B)} \log\frac{1}{\eps_\bx}}% \text{ calls of } \nabla F(\bx) \text{ and communications}.
\end{align*}
calls of $\nabla f_i(\cdot)$ at each node and communication rounds, with $\eps_\bx$ being the desired distance to the solution: $\|\bx^N -\bx^*\| \leq \eps_\bx$.
In the smooth, strongly convex case it is also the complexity for satisfying $F(\bx^N) - F(\bx) \leq \eps_F$ or $ \|\mA \bx^N\| \leq \eps_\mA$ (up to logarithmic dependencies on the problem parameters).
Indeed, from Lipschitz smoothness we have $F(\bx^N) - F(\bx) \leq L\eps^2_\bx / 2$ and 
$\|\mA \bx^N\| = \|\mA \bx^N -\mA \bx^*\| \leq \sigma_{max}(\mA)\eps_\bx$.
By that means, in the following inequalities $\eps$ can be replaced by any of $\eps_\bx,~ \eps_f, ~ \eps_\mA$.

The dependence on network parameters $W$ and affine constraints parameters $B$ can be enhanced by using Chebyshev acceleration \cite{scaman2017optimal}. Let us replace $W$ by a Chebyshev polynomial $P_{K}(W)$ such that it has degree $K = O\cbraces{\sqrt{\chi(W)}}$ and condition number $\chi(P_K(W)) = O(1)$. Multiplication by $P_K(W)$ is equivalent to making $K$ communication rounds. Analogically, let us replace $B^\top B$ by a Chebyshev polynomial $P_M(B^\top B)$ with degree $M = O\cbraces{\sqrt{\chi(B^\top B)}}$ and condition number $\chi\cbraces{P_M(B^\top B)} = O(1)$. As a result, we obtain
\begin{align*}
    &N = O\cbraces{\sqrt{\frac{L}{\mu}} \log\frac{1}{\varepsilon}} \text{ oracle calls at each node}, \\
    &O\cbraces{N\sqrt{\chi(W)}} \text{ communications}, \\
    &O\cbraces{N\sqrt{\chi(B^\top B)}} \text{ multiplications by $B$, $B^\top$ at each node}.
    % &O\cbraces{\sqrt{\frac{L}{\mu}} \sqrt{\chi(W)} \log\frac{1}{\varepsilon}} \text{ communications}, \\
    % &O\cbraces{\sqrt{\frac{L}{\mu}} \sqrt{\chi(B^\top B)} \log\frac{1}{\varepsilon}} \text{ multiplications by $B$, $B^\top$}.
\end{align*}

\section{Globally Dual Approach}\label{sec:globally_dual_approach}

In this section, we describe an approach to solving \eqref{eq:problem_constraints} that is based on passing to the dual problem. We call this approach ``global'' since both constraints, that is, affine constraints $\mB\bx = 0$ and communication constraints $\mW\bx = 0$ are used in the dual reformulation.

Let $\gamma$ be a positive scalar and $\mA^\top = [\mB^\top~ \gamma\mW]$ and introduce dual function
\begin{align*}
	\Phi(\by) = \max_{\bx\in\R^{md}} \sbraces{-F(\bx) + \angles{\by, \mA\bx}} = F^*(\mA^\top\by).
\end{align*}
We have $\ds\nabla\Phi(\by) = \mA\nabla F^*(\mA^\top\by) = \mA\cdot\argmin_{\bx\in\R^{md}}\sbraces{-F(\bx) + \angles{\by, \mA\bx}}$. Note that multiplication by $\mA$ is performed in a distributed manner: indeed, it includes local multiplications by $B$ and a consensus round, which is a multiplication by $\mW$. Moreover, the $\argmin$ operation is computed locally, which is standard for decentralized optimization \cite{scaman2017optimal}. Finally, dual function $\Phi$ is $\frac{\lambda_{\max}(\mA^\top\mA)}{\mu}$-smooth on $\R^{m(p+d)}$ and $L_\Phi =\frac{\lambda_{\min}^+(\mA^\top\mA)}{L}$-strongly convex on $(\kernel\mA^\top)^\bot$. Solving dual problem
\begin{align*}
	\min_{\by\in\R^{m(p+d)}}~ \Phi(\by)
\end{align*}
by a fast gradient method (see i.e. accelerated Nesterov method in Section 2.2 of \cite{nesterov2004introduction}) until accuracy 
$\Phi(\by^N) - \Phi(\by) \leq \eps_\Phi$ requires 
$N = O\cbraces{\sqrt\frac{L}{\mu}{\sqrt{\chi(\mA^\top\mA)}}\log\frac{1}{\eps_\Phi}}$ iterations.

Following the same arguments as in Section \ref{sec:primal_approach}, we compute the condition number $\chi(\mA^\top\mA)$:
\begin{align*}
    \chi(\mA^\top\mA) = \frac{\lambda_{\max}(B^\top B) + \gamma^2\lambda_{\max}^2(W)}{\min\braces{\lambda_{\min}^+(B^\top B), \gamma^2(\lambda_{\min}^+(W)^2)}}.
\end{align*}
The minimal value of $\chi(\mA^\top\mA)$ is attained at $\gamma^2 = \frac{\lambda_{\min}^+(B^\top B)}{(\lambda_{\min}^+(W))^2}$ and equals $\chi(\mA^\top\mA) = \chi(B^\top B) + \chi^2(W)$.
Communication and computation complexities of fast dual method equal
\begin{align*}
	O\cbraces{\sqrt\frac{L}{\mu} \cbraces{\chi(B^\top B) + \chi^2(W)}^{\color{black} \frac12} \log\frac{1}{\eps_\Phi}}.
\end{align*}

To obtain desired complexity estimates for the algorithm to find the approximate solution $\bx^N$ satisfying
$F(\bx^N) - F(\bx) \leq \eps$ and 
$ \|\mA \bx^N\| \leq \eps$, we refer to the following properties of dual function (see, e.g. Theorem 5.2 from \cite{gorbunov2019optimal}):
\begin{align*}
    \|\nabla \Phi(\by) \| \leq \epsilon / R_\by 
    &\Rightarrow F(\bx(\by)) - F(\bx^*) \leq \epsilon, \\
    \|\nabla \Phi(\by)\| \leq \epsilon 
    &\Rightarrow \|\mA \bx(\by) \| \leq \epsilon,
\end{align*}
where $\|\by\| \leq 2 R_\by$, and 
$\bx(\by) = \argmin_{\bx\in\R^{md}}\sbraces{-F(\bx) + \angles{\by, \mA\bx}}$.
Combining it with $\Phi(\by^N) - \Phi(\by) \geq \|\Phi(\by^N)\|^2 / 2L_\Phi$, which is true for a smooth convex function, we justify substitution of $\eps_\Phi$ by $\eps$ in the complexity estimate.
This transition will only change the constant hidden by big-O notation (by the factor of two), and affect omitted logarithmic dependencies on the problem parameters.

To employ Chebyshev acceleration in this case we do substitution $\mA^\top \by \to \bp$.
In this variables accelerated Nesterov method turns into Algorithm \ref{alg:globally_dual}, where $\bx(\bq) = \nabla F^*(\bq) = \argmin\sbraces{-F(\bx) + \angles{\bq, x}}$:
\begin{algorithm}[H]
	\caption{Globally Dual Method}
	\label{alg:globally_dual}
	\begin{algorithmic}[1]
		\STATE {\bf Input:} $\bp^0 \in \range \mA^\top$, $\eta>0$, $\beta \in (0, 1)$
		\STATE $\bp^{-1} = \bp^0$
		\FOR {$k=0,1,2,\ldots$}
		\STATE $\bq = \bp^k + \beta\cbraces{\bp^k - \bp^{k-1}}$
		\STATE $\bp^{k+1} = \bq - \eta \mA^\top \mA \bx(\bq)$
		\ENDFOR
	\end{algorithmic}
\end{algorithm}
For the algorithm in this form we can replace $\mA^\top \mA$ with Chebyshev polynomial of it, as we did in Section \ref{sec:primal_approach}, and obtain the same complexity estimates as for APDG:
% Analogically to Section \ref{sec:primal_approach}, employing Chebyshev acceleration results in the following complexities:
\begin{align*}
    &N = O\cbraces{\sqrt{\frac{L}{\mu}} \log\frac{1}{\varepsilon}} \text{ oracle calls at each node}, \\
    &O\cbraces{N\sqrt{\chi(W)}} \text{ communications}, \\
    &O\cbraces{N\sqrt{\chi(B^\top B)}} \text{ multiplications by $B$, $B^\top$ at each node}.
\end{align*}

\section{Locally Dual approach}\label{sec:locally_dual_approach}

In Section \ref{sec:globally_dual_approach} we discussed a dual reformulation of \eqref{eq:problem_constraints} where both constraints $\mB\bx = 0$ and $\mW\bx = 0$ are used simultaneously. This section describes a dual approach, as well, but the difference is that we firstly pass to dual functions locally at the nodes and impose the communication constraints only afterwards.

\subsection{Utilizing locality on $\bu$}
One can note that in the above approaches optimization over $\bu$ could be done locally at each node.
This is equivalent to including affine constraints into the objective (as an indicator function) instead of handling them with Lagrangian multipliers.
In settings there the ``cost'' of communication is limiting or comparable to that of local computations, we can find the solution faster by going this way.
It may be the case when $x$ has a small dimension and decentralization is desirable due to privacy constraints.

% \begin{align*}
%     \min_{\mB \bx = 0} &\sum_{i=1}^m f_i(x_i) \\
%     \text{s.t. } &\mW\bx = 0\\ 
% \end{align*}

Dual problem in this approach will be
\begin{equation*}
    \max_\bv \min_{\mB \bx=0} \{F(\bx) + \langle \bv, \mW \bx \rangle  \} = - \min_\bv F^*_{[\mB\bx=0]}(\mW^\top \bv),
\end{equation*}
where $\ds F^*_{[\mB\bx=0]}(\bv) = \max_{\mB \bx = 0} \{ \langle \bv , \bx \rangle - F(\bx)\}$~ denotes a convex conjugate under affine constraints.

We can reduce the problem of computing the gradient of such a modified conjugate function to calling conventional dual oracle.
Let $E$ be a matrix, the rows of which constitute an orthogonal basis in the null space of $B$ (matrix $E$ can be computed at the preprocessing stage of an optimization algorithm).
Then instead of working with functions $f_i(x)$ we can optimize the sum of functions $h_i(t) = f_i(Et)$.

Denote $\bt = \col(t_1, \ldots, t_m)$, $H(\bt) = \sum_{i=1}^m h_i(t_i)$. Then problem \eqref{eq:problem_initial} could be written in decentralized way as follows

\begin{align*}\label{eq:problem_t}
    \min_\bt &\sum_{i=1}^m h_i(t_i) \\
    \text{s.t. } &\mW_\bt\bt = 0.\\ 
\end{align*}

Its dual form is
\begin{align*}
\max_{\bt} \{ \langle \bz , \mW_\bt\bt \rangle - H(\bt)\}  = - \min_{\bz} H^*(\mW_\bt^\top \bz),
\end{align*}

and the gradient of the objective can be computed using Demyanov–Danskin's theorem:
\begin{align*}
\nabla H^*(\bz) = \arg\max_\bt \{ \langle \bz , \bt \rangle - F(\mE\bt)\} .
\end{align*}
From smaller dimension of $t$ comparing to $x$ we can expect that computation of $\nabla H^*(\bz)$ is easier than calling conventional first-order dual oracle, the only drawback is the necessity of storing matrix $E$ and performing multiplications by $E$.

Let $\mu_{t}$ and $L_{t}$ be the constants of strong convexity and Lipschitz smoothness of $h_i$ respectively for all $i=1,\ldots,m$.
Then, obviously, $\mu_t \geq \mu$ and $L_t \leq L$.
For example, if $f_i(x)$ is twice continuously differentiable, then its smoothness constant can be computed as $\ds L_{x, i} = \sup_{x \in \R^n}\lambda_{\max}(\nabla^2 f_i(x))$. The smoothness constant of $h_i(t)$ is given by 
$\ds L_{t, i} = \sup_{t \in \R^{d_t}}\lambda_{\max}(E^\top \nabla^2 f_i(Et) E)$. Note that the dimension of $t$ can be computed as $d_t = d - \text{rank}(B)$.
In the latter variant the maximum is taken over a smaller set of points, and multiplication by $E$ is likely to further reduce the smoothness constant (and increase strong convexity constant).

% The decrease in number of communication steps here is explained by three factors: 
% \begin{enumerate}
%     \item Lower variable dimension,
%     \item  Better smoothness / convexity properties of $h_i(t)$ comparing to $f_i(x)$.
% Let $\mu_{t}$ and $L_{t}$ be the constants of strong convexity and Lipschitz smoothness of $h_i$ respectively for all $i=1,\ldots,m$.
% Then, obviously, $\mu_t \geq \mu$ and $L_t \leq L$.
% For example, if $f_i(x)$ is twice continuously differentiable, then it is well-known that its smoothness constant 
% $L_{x, i} = \sup_{x \in \R^n}\lambda_{\max}(f_i''(x))$.
% Compare it with smoothness constant of $h_i(t)$ given by 
% $L_{t, i} = \sup_{t \in \R^{n_t}}\lambda_{\max}(E^\top f_i''(Et) E)$, where 
% $d_t = \dim(t) 
% = d - \text{rank}(B)$.
%     \item Independence on spectrum of $B$. 
% \end{enumerate}

Since $H(\bt)$ is $L_t$-smooth and $\mu_t$-strongly convex, we have that $F^*_{[\mB\bx=0]}(\bz) = H^*(\bz)$ is $\frac1{\mu_t}$-smooth and $\frac{1}{L_t}$-strongly convex \cite{kakade2009on}.

Thus, the fast gradient method \cite{nesterov1983method} applied to the dual problem requires 
\begin{align*}
    O\left(\sqrt{\frac{L_t}{\mu_t}} \chi (W) \log \frac{1}{\eps} \right),
\end{align*}
dual-oracle calls and communication rounds to ensure
$F(\bx^N) - F(\bx) \leq \eps$ and 
$ \|\mA \bx^N\| \leq \eps$ (see Section \ref{sec:globally_dual_approach} for details).
And using Chebyshev acceleration as described in Section \ref{sec:primal_approach} we can reduce the complexities to

\begin{align*}
    &N = O\cbraces{\sqrt{\frac{L_t}{\mu_t}} \log\frac{1}{\varepsilon}} \text{ oracle calls at each node}, \\
    &O\cbraces{N \sqrt{\chi(W)}} \text{ communications}. \\
    % &O\cbraces{\sqrt{\frac{L_t}{\mu_t}} \sqrt{\chi(W)} \log\frac{1}{\varepsilon}} \text{ communications}. \\
\end{align*}

%     \begin{algorithm}[h]
% \caption{{\tt PDSTM}}
% \label{Alg:PDSTM}   
%  \begin{algorithmic}[1]
% \STATE {\bf Input:} $\tilde{y}^0=z^0=y^0=0$, number of iterations $N$, $\alpha_0 = A_0=0$
% \FOR{$k=0,\dots, N$}
%   \STATE Set $\alpha_{k+1} = \frac{1}{2L_{\dm{\psi}}} + \sqrt{\frac{1}{4L_{\dm{\psi}}^2}+\frac{A_k}{L_{\dm{\psi}}}}$, $A_{k+1} = A_k + \alpha_{k+1}$
% \STATE $\tilde{y}^{k+1} = (A_ky^k+\alpha_{k+1}z^k)/A_{k+1}$
% \STATE $z^{k+1} = z^k - \alpha_{k+1} \nabla \psi(\tilde{y}^{k+1}) = z^k - \alpha_{k+1} A x(A^T\tilde{y}^{k+1})$
% \STATE $y^{k+1}=(A_ky^k+\alpha_{k+1}z^{k+1})/A_{k+1}$

% \ENDFOR
%         \STATE \textbf{Output:}    $y^N$, $x^N = \frac{1}{A_N}\sum_{k=0}^N \alpha_k x(A^T\tilde{y}^k)$
% \end{algorithmic}
%  \end{algorithm}

\section{Numerical Experiments}\label{sec:numerical_exp}
In the simulation we consider the following smooth, strongly convex objective function:
\begin{align*}
f_i(x) &= \frac{1}{2}\|C_i x - d_i\|_2^2+\frac{\theta}{2}\|x\|_2^2,\\
% \end{align*}
% % A random undirected connected graph G is taken.   
% \begin{align*}
F(\bx) &= \frac{1}{2}\|\mC \bx - \bd \|_2^2+\frac{\theta}{2}\|\bx\|_2^2,\\
\mC &= \text{diag}(C_1, \ldots, C_m),~ \bd = \col{(d_1, \ldots,  d_m)}.
\end{align*}
% Matrix W, which is associated with this graph, has a shape (18,18).
We consider different parameters of the problem such as the dimension of $x$, the rank of $B \in \R^{\dim(x) \times \dim(x)}$ and the number of nodes.
For each case we plot function error and constraints violation norm at each iteration for all our algorithms: APDG, Locally and Globally Dual approaches.
The Chebyshev acceleration is not applied in the experiments, so each iteration corresponds to one gradient computation (gradient of primal function in case of APDG, and gradient of dual function is case of dual approaches).
We also provide tables with comparison of time and number of iterations required to achieve given accuracy.
Time is measured with our Python/NumPy \cite{harris2020array} implementation of the algorithms, which is available on GitHub\footnote{Source code: \href{https://github.com/niquepolice/decentr_constr_dual}{https://github.com/niquepolice/decentr\_constr\_dual}}. 

\begin{enumerate}
\item For the first case we consider the ring network with $m=5$ nodes, $x \in \R^{40}$ and $\rank B = 1$.
Typical convergence plot is shown on Fig. \ref{fig:n5_d40}. 
One can see that all algorithms converge linearly, with the fastest one in terms of iterations number being Locally Dual, and the slowest one being APDG.
However, computing the gradient of a dual function might be an arithmetically more expensive operation than computing primal gradient in the black-box scenario.
In our implementation we compute the gradient of dual function by numerically solving the system of linear equations with its right-hand part being changed between iterations.
It means that one iteration of the Dual methods is more time-consuming than one iteration of APDG. 
In the Table \ref{tab:n5_d40}, we compare computational time and number of iterations required to achieve given accuracy. The results are averaged for 100 randomly generated problems. 

% First of all, we minimize the sum of $f_i(x)$ and find   $x^*$ and $f^*$. Then each agent runs an algorithm 1 from ~\cite{kovalev2021accelerated}.  The value of the function at the point, which was obtained in the algorithm, is calculated at each node. To obtain the $ \sum\limits_{i=1}^m (f_i(x^*))$ we solve the linear system by the least squares method. 

% \begin{figure}[H]
% \begin{minipage}[t]{0.45\linewidth}
% \centering
% \includegraphics[height=4cm,width=4cm]{5.png}
% \caption{ }
% \end{minipage}%
% \begin{minipage}[t]{0.45\linewidth}
% \centering
% \includegraphics[height=4cm,width=4cm]{2.png}
% \caption{ }
% \end{minipage}
% \begin{minipage}[t]{0.45\linewidth}
% \centering
% \includegraphics[height=4cm,width=4cm]{3.png}
% \caption{ }
% \end{minipage}%
% \begin{minipage}[t]{0.45\linewidth}
% \centering
% \includegraphics[height=4cm,width=4cm]{1.png}
% \caption{ }
% \end{minipage}
% \end{figure}
\begin{figure}
    % \vspace{-0.5cm}
    \centering
    \includegraphics[width=0.9\textwidth]{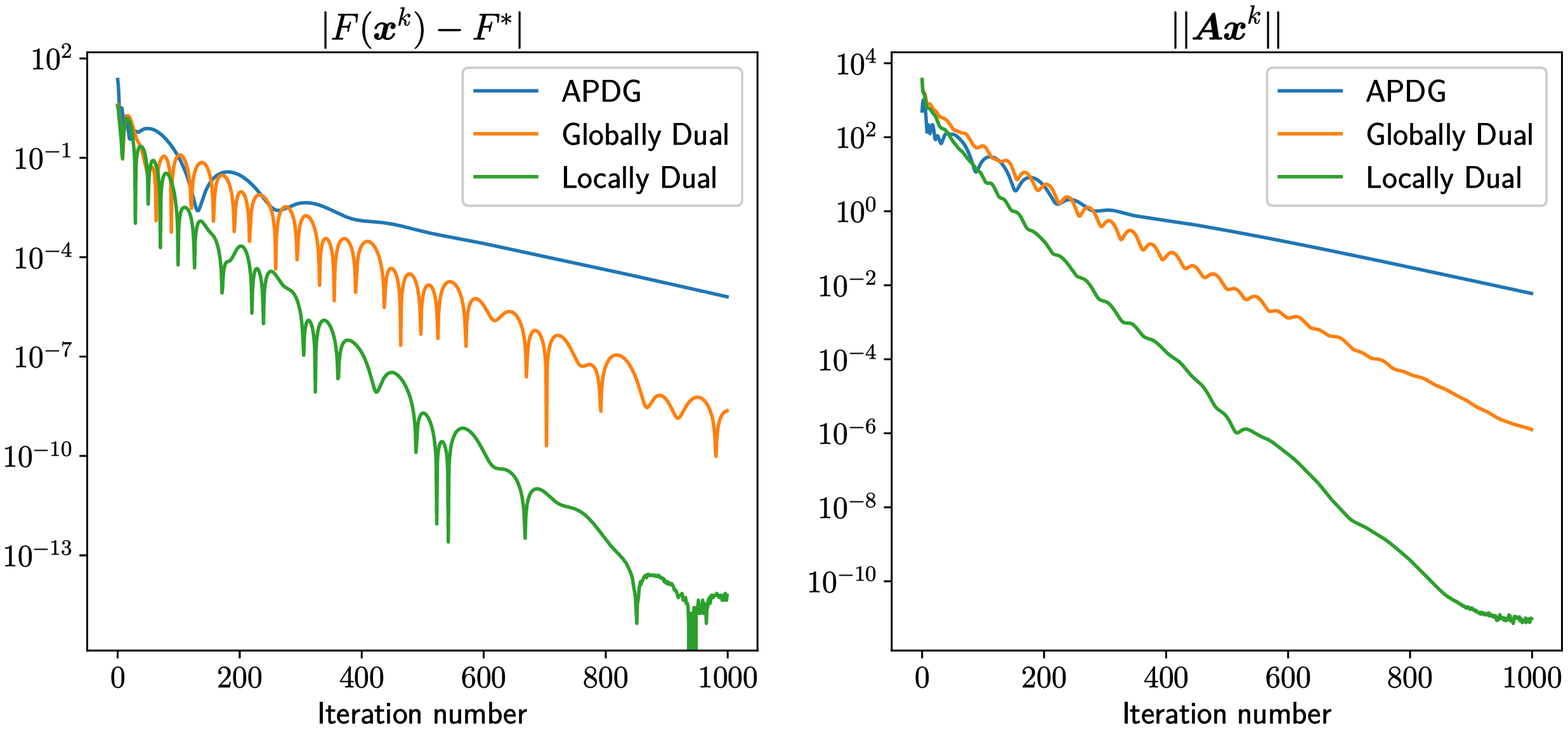}
    % \vspace{-0.3cm}
    \caption{5 nodes, $\dim(x)=40$, $\rank B = 1$.}
    \label{fig:n5_d40}
\end{figure}

\begin{table}
\centering
\caption{Time and iterations for achieving  $\| \mA \bx^k \| < 10^{-2}$. Averaged over 100 experiments. Problem parameters: 5 nodes, $\dim(x)=40$, $\rank B = 1$.}
 \label{tab:n5_d40}
\begin{tabular}{ |c|c|c|c| } 
 \hline
 & APDG & Globally Dual & Locally Dual \\ 
 \hline
%  Iterations & 292.3 & 214.7 & 126.1 \\ 
%  Time (s) & 0.090 & 0.244 & 0.135 \\ 
 Iterations & 875.3 & 502.7 & 276.7 \\ 
 Time (s) & 0.193 & 0.510 & 0.233 \\ 
 \hline
\end{tabular}
\end{table}

%  The results of numerical experiment are given in Fig. 1-4(log scale). We plot the   value of $\sum\limits_{i=1}^m (f_i(x^*))-\sum\limits_{i=1}^m (f_i(x))$, $i =1 , \ldots, m$, against the number of communication rounds(with different condition numbers). 

\item Next we use the same number of nodes and the dimension of $x$, but increase the rank of $B$.
Even for $\rank B = 3$ the condition number of the locally dual problem usually is about two orders of magnitude smaller than the condition number of the globally dual problem, therefore the globally dual approach  has a significant advantage in that case.
Typical convergence plots are shown in Figure \ref{fig:n5_d40_rk3}, averaged iteration and time complexities for satisfying stopping criteria  are shown in Table \ref{tab:n5_d40_rk3}.
\begin{figure}
    \centering
    \includegraphics[width=0.9\textwidth]{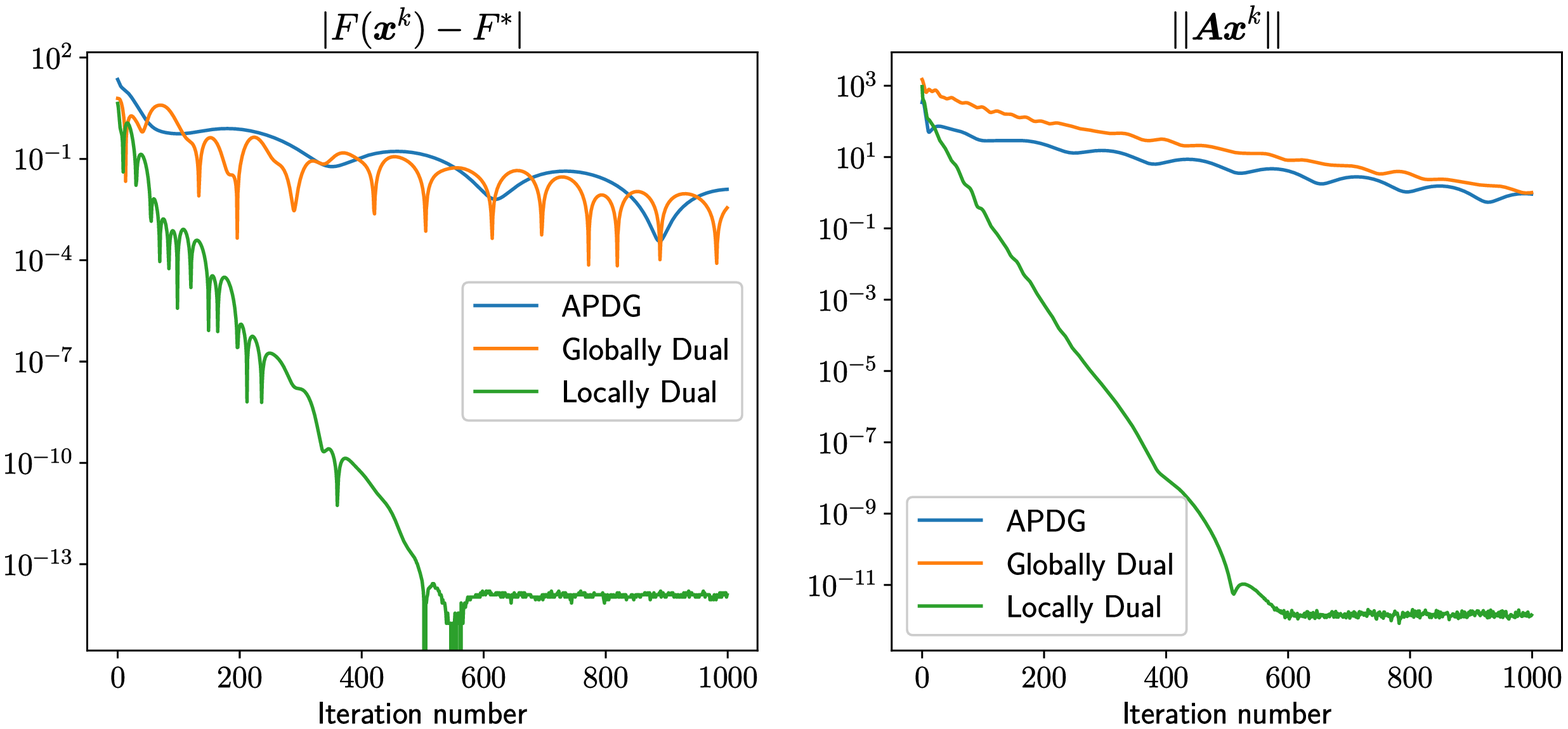}
    \caption{5 nodes, $\dim(x)=40$, $\rank B = 3$.}
    \label{fig:n5_d40_rk3}
\end{figure}

\begin{table}
\centering
\caption{Time and iterations for achieving  $\| \mA \bx^k \| < 10^{-1}$.
Averaged over 100 experiments.
Problem parameters: 5 nodes, $\dim(x)=40$, $\rank B = 3$.}
 \label{tab:n5_d40_rk3}
\begin{tabular}{ |c|c|c|c| } 
 \hline
 & APDG & Globally Dual & Locally Dual \\ 
 \hline
%  Iterations & 828.7 & 577.6 & 87.0 \\ 
%  Time (s) & 0.218 & 0.633 & 0.102 \\ 
 Iterations & 1555.5 & 1551.7 & 123.1 \\ 
 Time (s) & 0.337 & 1.577 & 0.127 \\ 
 \hline
\end{tabular}
\end{table}

\item In the case of higher dimension ($10$ nodes, $\dim(x) = 100$, $\rank B = 1$) we used Erd\H{o}s-R\'enyi random communication graphs with edge probability = $0.3$.
APDG seems to converge much faster by constraints violation norm at first iterations then other methods (Fig. \ref{fig:n10_d100}), and its convergence rate is close to other methods.
See also Table \ref{tab:n10_d100} for averaged results of multiple experiments.

\begin{figure}
    \centering
    \includegraphics[width=0.9\textwidth]{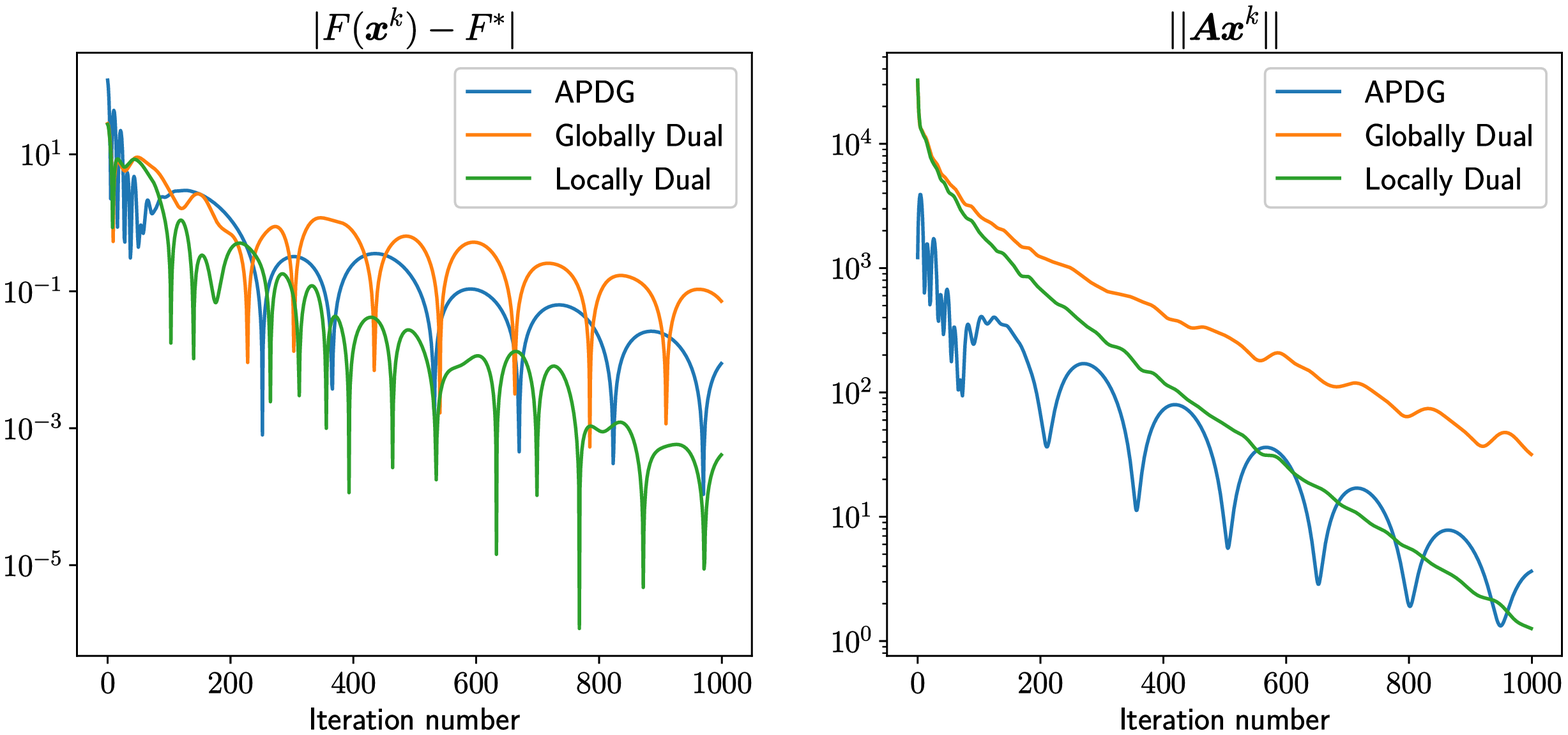}
    \caption{Erd\H{o}s-R\'enyi graph on 10 nodes, average degree = 3.6. $\dim(x)=100$, $\rank B = 1$.}
    \label{fig:n10_d100}
\end{figure}
 \begin{table}
\centering
\caption{Time and iterations for achieving accuracy $\| \mA \bx^k \| < 10^{1}$.
Averaged over 10 experiments. Problem parameters: 10 nodes, edge probability = $0.3$, $\dim(x)=100$, $\rank B = 1$.}
 \label{tab:n10_d100}
\begin{tabular}{ |c|c|c|c| } 
 \hline
 & APDG & Globally Dual & Locally Dual \\ 
 \hline
%  Iterations & 828.7 & 577.6 & 87.0 \\ 
%  Time (s) & 0.218 & 0.633 & 0.102 \\ 
 Iterations & 404.3 & 2227.9 & 1425.5 \\ 
 Time (s) & 2.561 & 54.024 & 16.544 \\ 
 \hline
\end{tabular}
\end{table}
\end{enumerate}

\bibliographystyle{abbrv}
\bibliography{references}

\end{document}

%% file: header.tex
\usepackage[english]{babel}
\usepackage{hyperref}       % hyperlinks
\usepackage{url}            % simple URL typesetting
\usepackage{booktabs}       % professional-quality tables
\usepackage{amsfonts}       % blackboard math symbols
\usepackage{nicefrac}       % compact symbols for 1/2, etc.
\usepackage{microtype}      % microtypography
\usepackage{lipsum}
\usepackage{mathtools}

\usepackage{amsmath,amssymb}
\usepackage{algorithm,algorithmic}
\usepackage{pifont}
\usepackage{cases}
\usepackage{subcaption,graphicx}
\usepackage{stackengine}    % circled symbols
\usepackage{wrapfig}

\newtheorem{assumption}[theorem]{Assumption}

\newcommand{\eps}{\varepsilon}

\newcommand{\one}{\mathbf{1}}

\DeclareMathOperator*{\argmin}{arg\,min}

\DeclareMathOperator{\spn}{span}
\DeclareMathOperator{\col}{col}
\DeclareMathOperator{\kernel}{Ker}
\DeclareMathOperator{\range}{Range}

\DeclareMathOperator{\rank}{rank}

\newcommand{\R}{\mathbb{R}}

\newcommand{\mA}{{\bf A}}
\newcommand{\mB}{{\bf B}}
\newcommand{\mC}{{\bf C}}

\newcommand{\mE}{{\bf E}}

\newcommand{\mI}{{\bf I}}

\newcommand{\mW}{{\bf W}}

\newcommand{\cE}{{\mathcal{E}}}

\newcommand{\cG}{{\mathcal{G}}}

\newcommand{\cV}{{\mathcal{V}}}

\newcommand{\bd}{{\bf d}}

\newcommand{\bp}{{\bf p}}
\newcommand{\bq}{{\bf q}}

\newcommand{\bt}{{\bf t}}
\newcommand{\bu}{{\bf u}}
\newcommand{\bv}{{\bf v}}

\newcommand{\bx}{{\bf x}}
\newcommand{\by}{{\bf y}}
\newcommand{\bz}{{\bf z}}

\newcommand{\ds}{\displaystyle}
\newcommand{\norm}[1]{\left\| #1 \right\|}

\newcommand{\angles}[1]{\left\langle #1 \right\rangle}
\newcommand{\cbraces}[1]{\left( #1 \right)}
\newcommand{\sbraces}[1]{\left[ #1 \right]}
\newcommand{\braces}[1]{\left\{ #1 \right\}}

\usepackage[colorinlistoftodos,bordercolor=orange,backgroundcolor=orange!20,linecolor=orange,textsize=scriptsize]{todonotes}

\def\R{\mathbb R}

\def\one{{\mathbf 1}}

\def\<#1,#2>{\langle #1,#2\rangle}